\documentclass[11pt,reqno]{amsart} 
\usepackage{amssymb,latexsym}
\usepackage{graphicx}
\usepackage{fancyhdr}

\def\cal{\mathcal}

\theoremstyle{plain}
\numberwithin{equation}{section}
\newtheorem{thm}{Theorem}[section]
\newtheorem{theorem}[thm]{Theorem}

\newtheorem{lemma}[thm]{Lemma}
\newtheorem{definition}[thm]{Definition}
\newtheorem{proposition}[thm]{Proposition}
\newtheorem{remark}[thm]{Remark}
\newtheorem{corollary}[thm]{Corollary}

\begin{document}

\overfullrule5pt
\def\sfrac#1#2{#1/#2}

\title[A sufficient condition]{A sufficient condition for cubic $3$-connected plane bipartite graphs to be hamiltonian}

\author{Jan~Florek}
\address{Faculty of Pure and Applied Mathematics,
 Wroclaw University of Science and Technology,
 Wybrze\.{z}e Wyspia\'nskiego 27,
50--370 Wroc{\l}aw, Poland}
\email{jan.florek@pwr.edu.pl}
\date{}

\begin{abstract}
Barnette's conjecture asserts that every cubic $3$-connected plane bipartite graph is hamiltonian. Although, in general, the problem is still open, some partial results are known. In particular, let us call a face of a plane graph big (small) if it has at least six edges (it has four edges, respectively). Goodey proved for a $3$-connected bipartite cubic plane graph $P$, that if all big faces in $P$ have exactly six edges, then $P$ is hamiltonian. In this paper we prove that the same is true under the condition that no face in $P$ has more than four big neighbours. We also prove, that if each vertex in $P$ is incident both with a small and a big face, then~$P$ has at least $2^{k}$ different Hamilton cycles, where $k = \left\lceil\frac{|B|-2}{4\Delta(B) - 7}\right\rceil$, $|B|$ is the number of big faces in $P$ and $\Delta(B)$ is the maximum size of faces in $P$.
\end{abstract}

\keywords{Barnette's conjecture, Eulerian plane triangulations, three colourable plane graphs, vertex partition into two induced forests}

\maketitle

\section{Introduction}
In this paper we consider only finite and simple graphs. We use Diestel \cite{flobar2} as a reference for undefined terms. In particular, if $G$ is a plane graph, then $ V(G)$ is the set of vertices and $\Delta (G)$ (or $\delta(G)$) is the maximum (minimum, respectively) degree of $G$. For $U \subseteq V(G)$, the neighbours in $V(G) \setminus  U$ of vertices in $U$ are called \textit{neighbours of} $U$.

Let $\cal P$ denote the family of all cubic $3$-connected plane bipartite graphs.  In 1969, Barnette  (see Tutte~\cite{flobar17}, Problem $5$) conjectured that every graph of ${\cal P}$ has a Hamilton cycle. The problem whether a cubic bipartite plane graph has a Hamilton cycle (without the assumption of $3$-connectivity) is NP-complete, as shown by Takanori, Takao and Nobuji \cite{flobar16}. Therefore, it seems wise to discuss only sufficient conditions for a graph of $\cal P$ to be hamiltonian. Goodey \cite{flobar7} proved that if a graph in ${\cal P}$ has only faces with $4$ or $6$ edges, then it is hamiltonian. Kelmans \cite{flobar12, flobar13}  proved that Barnette's conjecture holds if  and only if every cyclically-$4$-edge-connected graph of $\cal P$ has a Hamilton cycle (a graph is \textit{cyclically-k-edge connected} if  it has no edge cut of size $k-1$ such that both sides of the cut contain a cycle). Brinkmann, Goedgebeur and McKay \cite{flobar3} (see also Holton, Manvel and McKay \cite{flobar10}) used computer search to confirm Barnette's conjecture for graphs up to $90$ vertices. Florek \cite{flobar4} proved that if a graph of $\cal P$ possess a $2$-factor consisting only of facial $4$-cycles, then it has a Hamilton cycle.

Stein  \cite{flobar15} expresses hamiltonicity in terms of the dual graph (see also S.L. Hakimi and E.F. Schmeichel \cite{flobar8}).  Let $H$ be a plane triangulation. Stein proved that the dual $H^{*}$ has a Hamilton cycle if and only if there exists a partitioning of the vertex set of $H$ into two subsets each of which induces a forest. Notice that $G^{*} \in \cal P$ if and only if $G$ is an Eulerian plane triangulation different from $3$-cycle. In view of the Stein theorem Barnette's  conjecture is equivalent to the problem whether for each Eulerian plane triangulation there exists a partitioning of $V(G)$ into two subsets each of which induces a forest.

 A \textit{$k$-colouring} of a graph $G$ is a mapping $c \colon V(G) \rightarrow S$, where $S$ is a set of $k$ colours. Alternatively, a $k$-colouring may be viewed as a partition $\{V_{i}: i \in S\}$ of $V(G)$, where $V_{i}$ denotes the (possibly empty) set of vertices assigned to colour $i$.  A colouring~$c$ is \textit{proper} if no two adjacent vertices are assigned the same colour. Then, each $V_i$ forms an independent set of vertices. A graph is \textit{$k$-colourable} if it has a proper $k$-colouring. Note that every Eulerian plane triangulation has a unique proper $3$-colouring (see Heawood \cite{flobar9}).
 
Let $G$ be an Eulerian plane triangulation different from $3$-cycle. A vertex of $V(G)$ is called \textit{big} (\textit{small}) if it is of degree at least $6$ (of degree equals $4$, respectively). Let  $B$ denote the set of all big vertices in $G$ and suppose that $G[B]$ is a subgraph of $G$ induced by the set $B$. If $C$ is a cycle in $G$ whose vertices are all small in $G$, then there are tree possibilities: $G$ is the octahedron, or $C$ is of even order and $G$ is the join $C \vee E^2$ (where $E^{2}$ is an empty graph of order $2$), or $C$ is a  facial $3$-cycle such that neighbours of $V(C)$ induce a facial $3$-cycle in $G[B]$. If $G \neq  C^{2l} \vee E^2$, for $l \geqslant 2$, then each path $P$ in $G$ in which all inner vertices are small and ends are big is either an induced path in $G$ or $P + ac$ is an induced cycle in $G$, where $a$, $c$ are different adjacent ends of $P$. Notice that the set $V_{0}(P) $ of all neighbours of $V(P)$ consists of two big vertices in $G$ (say $b, d$) and $abcda$ is a facial $4$-cycle in $G[B]$. An easy verification shows that the following conditions are equivalent:
\begin{enumerate}
\item[($1$)] $G \neq  C^{2l} \vee E^2$, for $l \geqslant 2$,
\item[($2$)] $G$ has at least three big vertices,
\item[($3$)] every facial cycle in $G[B]$ is a $3$-cycle or a $4$-cycle,
\item[($4$)] $G[B]$ is $2$-connected.
\end{enumerate}
Let $\Theta$ be the family of all $2$-connected $3$-colourable plane graphs having only faces bounded by $3$-cycles or $4$-cycles. By the equivalence of conditions $(2)-(4)$, if $G$ is an Eulerian plane triangulation with  at least three big vertices, then $G[B] \in \Theta$. 

Let $J \in {\Theta}$ and suppose that  ${\cal J} =\{J_{1}, J_{2}, J_{3}\}$ is a partition of $V(J)$, where each $J_i$ is a (possibly empty) independent set in $J$.
\begin{definition}\label{definition 1.1} A partition  $\cal K$ of $V(J)$ into two subsets is called \textit{compatible} with~$\cal J$ if every facial $4$-cycle in $J$  has two non-consecutive vertices which belong either to the same part of $\cal K$ and ${\cal J}$ or to distinct parts of $\cal K$ and ${\cal J}$. 
\end{definition}
The following lemma asserts that using the partition $\cal J$ we can obtain many partitions of $V(J)$ into two subsets which are compatible with $\cal J$. 
We prove Lemma \ref{lemma1.1} in Chapter $3$. 
\begin{lemma} \label{lemma1.1}  If  $I \subseteq J_{1} \cup J_{2}$ is an independent set in $J$, then the partition  $\{(J_{1} \cup J_{2})\setminus  I, J_{3}\cup I\}$ of $V(J)$ is compatible with $\cal J$. 
\end{lemma}
Let $G$ be an Eulerian plane triangulation with at least three big vertices  and suppose that $\{V_{1}, V_{2}, V_{3}\}$ is the unique partition of $V(G)$ into three independent sets in $G$. Then ${\cal B} = \{ B \cap V_{1},  B \cap V_{2},  B \cap V_{3}\}$ is a partition of $B$ into three independent sets in $G[B]$. Notice that if $\{X, Y\}$ is a partition of $B$ such that $B \cap V_{1} \subseteq X$ and  $B \cap V_{2} \subseteq Y$ then, by Lemma~\ref{lemma1.1}, it is  compatible with $\cal B$.  Florek \cite{flobar5} (Proposition $2$) proved that if  both $X$ and $Y$ induce  a forest in $G[B]$ then $G^{*}$ is hamiltonian.  
The following theorem is a generalization of this result. It expresses hamiltonicity of  $G^{*}$  in terms of the graph $G[B]$. We prove Theorem \ref{theorem1.1} in Chapter $2$.
\begin{theorem}\label{theorem1.1} 
Let $G$ be an Eulerian plane triangulation with at least three big vertices.   The following conditions are equivalent:
\begin{enumerate}
\item[($a$)] $G^{*}$ is hamiltonian,
\item[($b$)] there exists a partition $\cal L$ of $B$ into two parts which is compatible with ${\cal B}$ and each part of $\cal L$ induces a forest in $G[B]$. 
\end{enumerate}
\end{theorem}
In Corollary  \ref{corollary2.1} of Chapter $2$ we give a simple application of Theorem \ref{theorem1.1}, which is a generalization of Florek\cite{flobar4} result. The following Proposition \ref{proposition1.1} is an extension of  Corollary  \ref{corollary2.1} (see also Florek \cite{flobar6}). The proof is contained in Chapter $2$.
\begin{proposition} \label{proposition1.1} Let $G$ be an Eulerian plane triangulation. If each facial $3$-cycle in $G$ has both a small and a big vertex, then $G^{*}$ has at least $2^{k}$ different Hamilton  cycles,  where $k =\left\lceil\frac{|B|-2}{4\Delta(G[B]) - 7}\right\rceil$.
\end{proposition}
We next prove the following theorem for some subfamily of  all $3$-connected graphs of $\Theta$.
\begin{theorem}\label{theorem1.2} Let $J \in {\Theta}$ be $3$-connected  and suppose that ${\cal J} =\{J_{1}, J_{2}, J_{3}\}$ is a partition of $V(J)$ into three independent sets. Assume that vertices of   $V(J)$ satisfy  the following conditions:
\begin{enumerate}
\item[($a_1$)] every vertex of $J_{1} \cup J_{2}$ is of degree at most $4$,
\item[($a_2$)] if $v \in J_1$ and $v_1$, $v_2$, $v_3$, $v_4$ are its consecutive neighbours in their natural cyclic order around $v$, such that $v_1$, $v_2 \in J_3$ and $v_3$, $v_4 \in J_2$, then $v_1$ or $v_2$ is of degree at most $4$. 
\end{enumerate}
 Then there exists an independent set $I \subseteq J_{1} \cup J_{2}$ such that ${\cal K} = \{(J_{1} \cup J_{2})\setminus  I, J_{3}\cup I\}$ is a partition of  $V(J)$ and each part of ${\cal K}$ induces a forest in $J$.
\end{theorem}
In view of Lemma \ref{lemma1.1} the partition $\{(J_{1} \cup J_{2})\setminus  I, J_{3}\cup I\}$ considered in Theorem~\ref{theorem1.2} is compatible with ${\cal J} =\{J_{1}, J_{2}, J_{3}\}$. Using Theorem \ref{theorem1.2} we  prove the following theorem for all graphs $J \in \Theta$ (not necessary $3$-connected) with  ${\Delta}(J)\leqslant 4$. The proofs of Theorem \ref{theorem1.2} and Theorem \ref{theorem1.3} are presented in Chapter $5$.
\begin{theorem}\label{theorem1.3} Let $J \in {\Theta}$ with ${\Delta}(J)\leqslant 4$ and suppose that ${\cal J}$  is a partition of $V(J)$ into three independent sets. Then there exists a partition ${\cal K}$ of  $V(J)$  into two subsets which is compatible with  $\cal J$ and each part of ${\cal K}$ induces a forest in $J$.
\end{theorem}
If $G$ is an Eulerian plane triangulation  different from $3$-cycle, then, by the equivalence of conditions $(1)-(2)$, $G = C^{2l} \vee E^2$, for some $l \geqslant 2$, or $G$ has at least three big vertices. Certainly, if $G = C^{2l} \vee E^2$, then $G^{*}$ is hamiltonian. Recall that if $G$ has at least three big vertices, then $G[B] \in  {\Theta}$. Hence, from Theorem \ref{theorem1.3} and Theorem~\ref{theorem1.1}, we obtain immediately the following main result.
\begin{theorem}\label{theorem1.4}
Let $G$ be an Eulerian plane triangulation different from $3$-cycle. Then $G^{*}$ is hamiltonian if no vertex in $G$ has more than four big neighbours. 
\end{theorem}
Jensen and Toft (\cite{flobar11}, Problem $2.12$) showed that Barnette's conjecture is true if and only if for every $ 3$-colourable plane graph $H$ there exists a vertex partition into two subsets each of which induces a forest in $H$. This is a consequence of  the Stein \cite{flobar15} result and the following Kr\'{o}l \cite{flobar14} result: a plane graph is $3$-colourable if and only if it is a subgraph of some Eulerian plane triangulation. Chen, Raspaud and Wang \cite{flobar1} proved that for every plane graph without intersecting $3$-cycles there exists such partition. In the following Proposition \ref{proposition1.2} we show that Barnette's conjecture is connected with the family $\Theta$. Its proof is contained in Chapter $3$.
\begin{proposition} \label{proposition1.2}
Barnette's conjecture is true if and only if for every $J \in {\Theta}$ and for each partition  $\cal J$ of $V(J)$ into three independent subsets (possibly one of which is empty) there exists a partition $\cal K$ of $V(J)$ into two subsets which  is compatible with $\cal J$ and each part of $\cal K$  induces a forest in $J$. 
\end{proposition}

\section{Hamiltonicity of  $G^{*}$  in terms of the graph $G[B]$}

Let $G$ be an Eulerian plane triangulation different than $3$-cycle and suppose that  $\{V_{ 1}, V_{2}, V_{3}\}$ is the unique partition of $V(G)$ into three independent sets in $G$. Vertices belonging to $V_i$ are called of \textit{type~$i$}. Let  $B$ be the set of all big vertices in $G$. Then ${\cal B} = \{ B \cap V_{1},  B \cap V_{2},  B \cap V_{3}\}$ is a partition of $B$ into three independent sets in $G[B]$. Let ${\cal P}_G$ denote the family of all paths $P$ in $G$  in which all inner vertices are small and ends are big. 

By the mentioned Stein result, in order to prove Theorem \ref{theorem1.1}, it is enough to prove the following:
\begin{theorem}\label{theorem2.1}
Let $G$ be an Eulerian plane triangulation with at least three big vertices. The following conditions are equivalent:
\begin{enumerate}
\item[($a$)] there exists a partition  of  $V(G)$ into two subsets each of which induces a forest in $G$,
\item[($b$)] there exists a partition $\cal L$ of $B$ into two parts which is  compatible with $\cal B$ and each part of $\cal L$ induces a forest in $G[B]$.
\end{enumerate}
\end{theorem}
\begin{proof}
$(a) \Rightarrow (b)$
 Assume that $\{U, W\}$ is a partition of $V(G)$ such that both $U$ and $W$ induce a forest in~$G$. Let vertices of $U$ (or $W$) be coloured by $\alpha$ ($\beta$, respectively). Certainly, both $B \cap U$ and $B \cap W$ induce a  forest in $G[B]$. We prove that a partition $\{B~\cap U,  B \cap W \}$ of $B$ is compatible with $\cal B$.

Let  $C = abcda$ be a facial $4$-cycle in $G[B]$ and suppose that  $a, c$ are ends of some path $P \in {\cal P}_{G}$ and $\{b, d\} = V_{0}(P)$. Notice that vertices $b, d$ are of the same type. Hence, if vertices  $b, d$ have the same colour, then they both have the same type and colour. If they have distinct colours, then vertices of $P$ have two distinct colours alternately. Then, vertices $a, c$ have the same colour  if and only if they both have the same type. 

$(b) \Rightarrow (a)$  Let $\{X, Y\}$ be a partition of $B$ into two subsets each of which induces a forest in $G[B]$  and assume that  $\{X, Y\}$ is compatible with  the partition $\cal B$.  Let  $t_{0}\colon B \rightarrow \{\alpha, \beta\}$  be a $2$-colouring such that $t_{0}(v) =\alpha$ for $v \in X$ and $t_{0}(v) =\beta$ for $v \in Y$. 

Notice that we can assume that $G$ has no facial $3$-cycles whose vertices are all small in~$G$. Hence, each small vertex of $G$ is an inner vertex of a path belonging to ${\cal P}_{G}$. Let  ${\cal R}_{G} =\{P_{1}, \ldots, P_{n}\}$ be the set of all paths of ${\cal P}_{G}$ satisfying the following additional condition:  for each pair of paths of length $2$ having a common inner vertex only one of them belongs to~${\cal R}_{G}$.  Suppose that $I(P_i)$ is the set of all inner vertices of $P_i$. Then, we have  
\begin{enumerate}
\item[($1$)] $I(P_i) \cap I(P_j) = \emptyset$ for every two different paths $P_{i}$ , $P_{j}\in {\cal R}_{G}$, 
\item[($2$)] $V(G) = B \cup I(P_{1}) \cup\ldots \cup I(P_{n})$.
\end{enumerate}
 Assume that $T_{0} = B$ and $T_{i} = T_{i-1} \cup I(P_i)$, for $i =1$, \ldots, $n$. In view of $(1)$,  $T_{i-1}$ and $I(P_i)$ are disjoint. Using this, we will define a sequence of $2$-colourings $t_{i}\colon T_{i} \rightarrow \{\alpha, \beta\}$, for $1 \leqslant i \leqslant n$, such that $t_i$ is an extension of $t_{i-1}$ on the set $I(P_i)$, and the set $X_i$ (or $Y_i$) of all $\alpha$-vertices ($\beta$-vertices) belonging to $T_i$ induces a forest in $G$. 
 
Fix $0 \leqslant i < n$. Assume that a $2$-colouring $t_{i}$ is defined. 
Let $P_{i+1} = v_{1}\ldots v_{k}$ and suppose that $abcda$ is a  facial $4$-cycle in $G[B]$ such that  $a = v_1$, $c = v_k$ are ends of $P_{i+1}$ and $\{b, d\} = V_{0}(P_{i+1})$.  Then, vertices $b, d$ are of the same type and vertices of $P_{i+1}$ are of two  different types alternately. 

Let $t_{i}(b) \neq t_{i}(d)$. Since $\{X, Y\}$ is compatible with $\cal B$, $t_{i}(v_1)  = t_{i}(v_k)$ if and only if vertices $v_1$, $v_k$ are of the same type. Hence, $t_{i}(v_1) = t_{i}(v_k)$ if and only if $k$ is odd. Let us define a $2$-colouring $t_{i+1}\colon T_{i+1} \rightarrow \{\alpha, \beta\}$ as follows: $t_{i+1}(v_j) = t_{i}(v_1)$ for $j$ odd, and $t_{i+1}(v_j) \neq t_{i}(v_1)$ for $j$ even. Hence, $t_{i+1}(v_k) = t_{i}(v_1)$ if and only if~$k$ is odd. Thus, $t_{i+1}(v_k) = t_{i}(v_k)$. So, the extension $t_{i+1}$ of $t_{i}$ is well defined. 

Let now $t_{i}(b) = t_{i}(d) = \alpha$ (or $\beta$). If there exists a path  connecting vertices $b$ and  $d$ which is contained in a subgraph induced by~$X_{i}$ (by $Y_i$, respectively) in $G$, we put $t_{i+1}(v) \neq t_{i}(b)$  for every $v \in I(P_{i+1})$. If no such path exists, we choose a vertex $v_{0} \in I(P_{i+1})$ and put $t_{i+1}(v_0) =  t_{i}(b)$ and $t_{i+1}(v)  \neq t_{i}(b)$ for every $v \in I(P_{i+1}) \setminus  \{ v_{0}\}$.

In both cases $X_{i+1}$ (or $Y_{i+1}$) induces a forest in $G$. Notice that $T_{n} = V(G)$, by $(2)$. Hence, implication $(b) \Rightarrow (a)$ holds which completes the proof.
 \end{proof}
  Let $C$ be a cycle in $J$. The bounded (unbounded) region of $R_{2} \setminus  C$  is called  \textit{interior} (\textit{exterior}, respectively) of $C$.  We denote them by $int(C)$ and $ext(C)$.
\begin{corollary} \label{corollary2.1}
If each $3$-cycle in $G$ has both a small and a big vertex, then $B(G)$ is bipartite and $G^{*}$ is hamiltonian. 
 \end{corollary}
 \begin{proof} If $G = C^{2l} \vee E^2$, for some $l \geqslant 2$, then $G^{*}$ is hamiltonian. Hence, we can assume that $G \neq  C^{2l} \vee E^2$, for $l \geqslant 2$. 
 
We first prove that $B(G)$ is bipartite. By equivalence of conditions $(1)$ and $(3)-(4)$, $G[B]$ is $2$-connected and every facial cycle in $G[B]$ is a $3$-cycle or $4$-cycle, where $B$ is the set of all big vertices of $V(G)$. If $\lambda$ is a facial $3$-cycle in $G[B]$, then it is not a facial $3$-cycle in $G$, because $\lambda$ contains a small vertex in $G$. Hence, if $\lambda$ is a facial $3$-cycle in $G[B]$, then $int(\lambda)$ (or $ext(\lambda)$) contains a small vertex of $G$. Then, vertices of $\lambda$ and vertices belonging to $int(\lambda)$ ($ext(\lambda)$, respectively) induce the octahedron in $G$. Thus, $G$ contains a facial $3$-cycle with three small vertices, which is a contradiction. It follows that each facial cycle in $G[B]$ is a $4$-cycle. Therefore $G[B]$ is bipartite. 

Since $G[B]$ is bipartite, there exists a partition $\cal L$ of $B$ into two independent subsets such that each facial cycle in $G[B]$ has two pairs of non-consecutive vertices belonging to the same class of $\cal L$. Hence,~$\cal L$ is compatible with ${\cal B}$. In view of Theorem \ref{theorem1.1},  $G^{*}$ is hamiltonian.
\end{proof}
\begin{proof} [\bf  Proof of Proposition \ref{proposition1.1}]
Since $B(G)$ is bipartite, there exists a $2$-colouring $t\colon B \rightarrow \{\alpha, \beta\}$. Assume first that $G[B]$ is not a $4$-cycle. Let $J$ be a graph such that $V(J)$ is the set of all facial cycles of $G[B]$ and any two vertices of $J$ (say facial cycles $f_{1}$ and $f_{2}$ in $G[B]$) are adjacent if and only if $f_{1}$ and $ f_{2}$ have a common vertex. Since each facial cycle of $G[B]$ is a $4$-cycle and $G[B]$ is not a $4$-cycle, we have $\Delta(J) \leqslant 4\Delta(B)-8$.
  
Notice that, by Euler's Formula, $J$ has $|B|-2$ vertices and, by the greedy algorithm, $\chi(J) \leqslant\Delta(J)+1$, where $\chi(J)$ is the chromatic number of $J$. Hence,  there exists an independent set of vertices $K \subset V(J)$ which has at least
$$|K| \geqslant \frac{|J|}{\chi(J)} \geqslant \frac{|B|-2}{\Delta(J)+1} \geqslant  \frac{|B|-2}{4\Delta(B)-7}$$
vertices. It follows, that there exists a family ${\cal K }= \{f_{1}, \ldots, f_{|K|}\}$ of facial $4$-cycles in $G[B]$ which are disjoint in pairs. For each $f_{i} \in {\cal K }$ there exists a path $P_{i} \in {\cal P}_{G}$ with ends $a_{i}, c_{i}$ and $V_{0}(P_{i}) = \{b_{i}, d_{i}\}$ such that $f_{i}= a_{i}b_{i}c_{i}d_{i}a_{i}$. Notice that $t(b_i) = t(d_i)$, for every $i =1$, \ldots,~$|K|$. If $t(b_i) = t(d_i) = \alpha$, we define two extensions $t_1$ and $t_2$ of $t$ on $I(P_i)$ in the following way: $t_{1}(v) = \beta$ for every vertex $v \in I(P_{i})$, $t_{2}(v) = \beta$ for every vertex $v \in I(P_{i}) \setminus \{v_0\}$ and $t_{2}(v_{0}) = \alpha$, where $v_0$ is a vertex of $I(P_i)$. It follows that there exist at least $2^{|K|}$ extensions of the colouring $t$ on $B \cup I(P_{1}) \cup\ldots \cup I(P_{|K|})$ such that for each of them the set of all $\alpha$-vertices ($\beta$-vertices) belonging to $B \cup I(P_{1}) \cup\ldots \cup I(P_{n})$ induces a forest in $G$. Since facial cycles of ${\cal K }$ are disjoint in pairs, there exist at least $2^{|K|}$ extensions of $t$ on $V(G)$ such that for each of them the set of all $\alpha$-vertices ($\beta$-vertices) belonging to $V(G)$ induces a forest in $G$. Hence, there exists at least~$2^{|K|}$ different Hamilton cycles in $G^{*}$. It is easy to see that if $G[B]$ is a $4$-cycle, then $G^{*}$ has at least~$10$ Hamilton cycles, because interior (exterior) of this cycle contains a path of the order at least $3$.
\end{proof}

\section{Partitions of $V(J)$ compatible with ${\cal J} = \{J_{1}, J_{2}, J_{3}\}$}
Let $J \in {\Theta}$ and suppose that  ${\cal J} = \{J_{1}, J_{2}, J_{3}\}$ is a partition of $V(J)$ into three independent sets.   Vertices belonging to $J_{i}$ are called of \textit{type $i$}. We first prove Lemma \ref{lemma1.1}. 

 \begin{proof} [\bf Proof of Lemma \ref{lemma1.1}]
Let $I \subseteq J_{1} \cup J_{2}$ be an independent set in $J$ and suppose that vertices of  $J_{3}\cup I$ (or $(J_{1} \cup J_{2})\setminus  I$) are coloured by $\alpha$ ($\beta$, respectively).  Let $C = abcda$ be a facial $4$-cycle in $J$. It is sufficient to prove that $C$ has two non-consecutive vertices which are of the same type and the same colour, or which are of distinct types and distinct  colours. Notice that $C$ has two non-consecutive vertices of the same type (say $a$ and $c$). Let us consider the following cases :
\begin{enumerate}
 \item[($i_1$)] vertices $a$, $c$  are of type $3$,
\item[($i_2$)] vertices $a$, $c$  are of type $1$,
\item[($i_3$)] vertices $a$, $c$  are of type $2$.
\end{enumerate}

Case $(i_1)$ Then $a$, $c$ are vertices of the same type and the  same colour.

Case $(i_2)$  If $a$, $c$ have distinct colours, then $a$, $c \notin J_{1} \setminus  I$. Hence one of vertices $a$, $c$ belongs to $I$. So, $b$, $d \in (J_{2}\setminus  I) \cup J_{3}$, because $I$ is independent in $J$ and $a$, $c$ are of type $1$. Hence, $b$ and $d$ both have the same type and  the same colour, or they are of distinct types and distinct colours.

The same proof as in Case $(i_2)$ remains valid for  Case $(i_3)$.
\end{proof}

We now prove the following two lemmas which will be used in the proof of Theorem \ref{theorem1.3}.
\begin{lemma}\label{lemma3.1} Assume that $P$ is a path in $J$  with ends of degree $3$, with all inner vertices of degree $4$ and such that its all vertices have exactly two common neighbours. If  $\{U, W\}$ is a partition of $V(J - P)$ which is compatible with $\{J_{1} \setminus  V(P), J_{2} \setminus  V(P), J_{3} \setminus  V(P)\}$ and both $U$ and $W$  induce a forest  in $J-P$, then there exists a partition $\cal K$ of $V(J)$ into two parts  which is compatible with $\cal J$ and each part of $\cal K$ induces a forest  in $J$.
\end{lemma}

\begin{proof}
Let $P$ be a path with ends (say $u$ and $w$) of degree $3$ in $J$, with all inner vertices of degree $4$ and such that its all vertices have exactly two common neighbours (say $b$ and $d$). Assume that $\{U, W\}$ is a partition of $V(J - P)$ into two parts  which is compatible with $\{J_{1} \setminus  V(P), J_{2} \setminus  V(P), J_{3} \setminus  V(P)\}$ and both $U$ and $W$  induce a forest  in $J-P$.  Let  $t\colon  V(J - P) \rightarrow \{\alpha, \beta\}$  be a $2$-colouring such that $t(v) =\alpha$ for $v \in U$ and $t(v) =\beta$ for $v \in W$.

Notice that vertices $b$ and $d$ are of the same type. Hence, they are not adjacent.  Then there exist the following facial $4$-cycles: $C = badcb$ in $J - P$, $C_{1} = badub$ in $J$ and $C_{2} = bwdcb$ in $J$. It is sufficient to extend $2$-colouring $t$ on $V(P)$ in such a way that the set of all $\alpha$-vertices and the set of all $\beta$-vertices of $V(J)$ induces a forest in $J$, and $C_1$ (or $C_2$) has two non-consecutive vertices which are either of the same type and the same colour or of distinct types and  distinct colours. Since $b$, $d$ are of the same type we can assume that they are of  distinct colours.  Let us consider the following cases:
\begin{enumerate}
\item[($i_1$)] vertices $a$, $u$, $w$ are of the same type, 
\item[($i_2$)] vertices $a$, $u$ are of  distinct types and vertices $u$, $w$ are of the same type,
\item[($i_3$)] vertices $a$, $u$ are of  the same type and vertices $u$, $w$ are of distinct types,
\item[($i_4$)] vertices $a$, $u$ are of  distinct types and vertices $u$, $w$ are of distinct types.
\end{enumerate}

Case $(i_1)$ We can assume that $a$, $b$ are coloured by $\alpha$ and  $d$ is coloured by $\beta$. We colour $u$ by $\alpha$ and vertices of $P$ by $\alpha$ and $\beta$ alternately. Then, vertices $a$, $u$, $w$ are of the same type and the same colour. Since  vertices $b$, $d$ are of the same type and distinct colours, vertices $a$, $c$ are of the same type and the same colour or they are of distinct types and distinct colours, because the partition $\{U, W\}$ of $V(J - P)$ is compatible with $\{J_{1} \setminus  V(P), J_{2} \setminus  V(P), J_{3} \setminus  V(P)\}$. If vertices $a$, $c$ are of the same type and the same colour (distinct types and distinct colours), then vertices $w$, $c \in V(C_2)$ are of the same type and the same colour (distinct types and distinct colours, respectively), because vertices $a$, $w$ are of the same type and the same colour. 

Case $(i_2)$ We can assume that $a$, $b$ are coloured by $\alpha$ and $d$ is coloured by $\beta$. We colour~$u$ by~$\beta$ and vertices of $P$ by $\beta$ and $\alpha$ alternately. Notice that vertices $a$, $u \in V(C_1)$ are of distinct types and distinct colours, and vertices $u$, $w$ are of the same type and the same colour. Hence, vertices $a$, $w$ are of distinct types and distinct colours. Since  vertices $b$, $d$ are of the same type and distinct colours, vertices $a$, $c$ are of the same type and the same colour or they are of distinct types and distinct colours, because the partition $\{U, W\}$ of $V(J - P)$ is compatible with $\{J_{1} \setminus  V(P), J_{2} \setminus  V(P), J_{3} \setminus  V(P)\}$. If vertices $a$, $c$ are of the same type and the same colour (distinct types and distinct colours), then vertices $w$, $c \in V(C_2)$ are of distinct types and distinct colours (the same type and the same colour, respectively), because vertices~$a$, $w$ have distinct types and distinct colours. 

The same proof as in Case $(i_1)$ (in Case $(i_2)$) remains valid for  Case $(i_3)$ (for Case $(i_4)$, respectively).
\end{proof}
\begin{lemma}\label{lemma3.2}   Assume that $J$ is not a cycle and  let $u \in V(J)$ be a vertex of degree $2$. If   $\{U, W\}$ is a partition of $V(J - u)$ which is compatible with $\{J_{1} \setminus  \{u\}, J_{2} \setminus   \{u\}, J_{3} \setminus   \{u\}\}$ and both $U$ and $W$ induce a forest in $J - u$,  then there exists a partition $\cal K$ of $V(J)$ into two parts which is compatible with  $\{J_{1}, J_{2}, J_{3}\}$ and each part of $\cal K$ induces a forest in $J$.
\end{lemma}

\begin{proof} 
Let $u$ be a vertex of degree $2$ in $J$. Since $J \in {\Theta}$ is not a cycle, $J-u \in {\Theta}$. Assume that $\{U, W\}$ is a partition  of $V(J - u)$ which is compatible with $\{J_{1} \setminus  \{u\}, J_{2} \setminus   \{u\}, J_{3} \setminus   \{u\}\}$ and both $U$ and $V$ induce a forest in $J - u$. Let  $t\colon  V(J - u) \rightarrow \{\alpha, \beta\}$  be a $2$-colouring such that $t(v) =\alpha$ for $v \in U$ and $t(v) =\beta$ for $v \in W$.

Let $b$, $d$ be two neighbours of $u$. Assume that $b$, $d$ are not adjacent (if vertices $b$, $d$ are adjacent the proof is analogous). Then there exist the following facial $4$-cycles: $C = badcb$ in $J - u$, $C_{1} = badub$ in $J$ and $C_{2} = budcb$ in $J$. It is sufficient to extend $2$-colouring $t$ on $\{u\}$ in such way that the set of all $\alpha$-vertices and the set of all $\beta$-vertices of $V(J)$ induces a forest in $J$, and $C_1$ (or~$C_2$) has two non-consecutive vertices which are either of the same type and the same colour or of distinct types and distinct colours. Let us consider the following cases:
\begin{enumerate}
 \item[($j_1$)] vertices $b$, $d$ are of the same type and distinct colours,
 \item[($j_2$)] vertices $b$, $d$ are of distinct types and the same colour.
 \end{enumerate}
 
 Case $(j_1)$ The same proof as in Cases $(i_1)-(i_2)$ of Lemma \ref{lemma3.1} remains valid for  Case $(j_1)$.

Case $(j_2)$  We can assume that vertices $b, d$ are coloured by $\alpha$. We colour $u$ by $\beta$. Since vertices $b$, $d$ are of distinct types, vertices $a$, $u$, $c$ are of the same type. Vertices $a$, $c$ are of the same colour, because the partition $\{U, W\}$ of $V(J - u)$ is compatible with  $\{J_{1} \setminus  \{u\}, J_{2} \setminus   \{u\}, J_{3} \setminus   \{u\}\}$.  Hence, $a$, $c$ are coloured by $\beta$ because $U$ induces a forest in $J - u$. Thus, vertices $a$, $u$, $c$ are of  the same type and the same colour, and the proof is complete.
\end{proof}

\begin{lemma}\label{lemma3.3} For every graph  $J \in \Theta$ with $\delta(J) \geqslant 3$ and for each partition $\cal J$ of $V(J)$ into three independent sets there exists an Eulerian plane triangulation $G$ such that  $J = G[B]$ and $\cal J = \{B \cap V_{1}, B \cap V_{2}, B \cap V_{3}\}$, where  $B$ is the set of all big vertices of $G$ and $\{V_{1}, V_{2}, V_{3}\}$ is the unique partition of $V(G)$ into three independent sets.
\end{lemma}

\begin{proof}
Let $J \in \Theta$ with $\delta(J) \geqslant 3$ and suppose that ${\cal J} = \{J_{1}, J_{2}, J_{3}\}$ is a partition of $V(J)$ into three independent sets. Let $s\colon V(J) \rightarrow \{1, 2, 3\}$ be a proper $3$-colouring of $V(J)$ such that $s(v) = k$ for $v \in J_{k}$, $k = 1$, $2$, $3$. Let $C_{i} = a_{i}b_{i}c_{i}d_{i}a_{i}$, for $i = 1$, \ldots, $n$, be the family of all facial $4$-cycles in $J$.   Certainly, each~$C_i$ has two vertices (say $b_i, d_i$) which have the same colour.  For each $C_i$, $i = 1$, \ldots, $n$, we add to~$J$ a path $P_i$ of length $2$ connecting $a_i$ and $c_i$ if $s(a_i) = s(c_i)$ (of length $3$ if $s(a_i) \neq s(c_i)$). Next we add edges $b_{i}v$ and $d_{i}v$ for each inner vertex $v$ of $P_i$. Further, suppose that $D_{i}= x_{i}y_{i}z_{i}x_{i}$,  for $i = 1,\ldots, m$, is the family of all facial $3$-cycles in~$J$. For each $D_i$, $i = 1$, \ldots,  $m$, we add to $J$ a $3$-cycle $x'_{i}y'_{i}z'_{i}x'_{i}$ and also a $6$-cycle $x_{i}x'_{i}y_{i}y'_{i}z_{i}z'_{i}x_i$. Certainly, by the above operations we can extend graph $J$ to a plane graph $G$ such that $V(J) = B$, where $B$ is the set of all big vertices of $G$ (because $\delta(J) \geqslant 3$). Moreover, $J = G[B]$ (because $P_{i}$ is not an edge). Certainly, we can extend the colouring $s$ to the proper $3$-colouring of $V(G)$ because $s(b_i) = s(d_i)$ and the path $P_i$ has length $2$  (length $3$) if $s(a_i)= s(c_i)$ ($s(a_i) \neq s(c_i)$, respectively), for $i = 1$, \ldots, $n$.  Hence, $G$ is an Eulerian plane triangulation and ${\cal J} = \{B \cap V_{1}, B \cap V_{2}, B \cap V_{3}\}$, where $\{V_{1}, V_{2}, V_{3}\}$ is the unique partition of $V(G)$ into three independent sets.
  \end{proof}

Now we are ready to prove Proposition \ref{proposition1.2}.

\begin{proof} [\bf  Proof of Proposition \ref{proposition1.2}]

 $(\Rightarrow)$ Assume that Barnette's conjecture is true. Let $J \in \Theta$ and suppose that ${\cal J} = \{J_{1}, J_{2}, J_{3}\}$ is a partition of $V(J)$ into three independent sets. We prove that the following condition is satisfied:
\begin{enumerate}
\item[$(i)$] there exists a partition $\cal K$ of $V(J)$ into two subsets which  is compatible with $\cal J$ and each part of $\cal K$  induces a forest in $J$.  
\end{enumerate}
 We apply induction on the order of the graph. Certainly, if $J$ is a $3$-cycle or $4$-cycle, then condition $(i)$ holds. 
 
 If $J$ is not a cycle and it contains a vertex of degree $2$, then $J - v \in {\Theta}$. By induction hypothesis there exists  a partition $\cal F$ of $V(J-v)$ into two parts which is compatible with $\{J_{1} \setminus  \{u\}, J_{2} \setminus   \{u\}, J_{3} \setminus   \{u\}\}$ and each part of  $\cal F$  induces a forest in $J-v$. Hence, by Lemma \ref{lemma3.2}, condition $(i)$ holds.

Assume now that $\delta(J) \geqslant 3$. By Lemma \ref{lemma3.3}, there exists an Eulerian plane triangulation $G$ such that $J = G[B]$ and $\cal J = \{B \cap V_{1}, B \cap V_{2}, B \cap V_{3}\}$, where  $B$ is the set of all big vertices of~$G$ and $\{V_{1}, V_{2}, V_{3}\}$ is the unique partition of $V(G)$ into three independent sets. Since $G^{*}$ is hamiltonian, from the  implication $(a) \Rightarrow (b)$ of Theorem \ref{theorem1.1}, condition $(i)$ holds.

$(\Leftarrow)$ Certainly, from the  implication $(b) \Rightarrow (a)$ of Theorem \ref{theorem1.1}  follows the backward implication of Proposition \ref{proposition1.2}.
\end{proof}

\section {A good colouring}

 Let $J \in {\Theta}$ be $3$-connected  and suppose that ${\cal J} =\{J_{1}, J_{2}, J_{3}\}$ is a partition of $V(J)$ into three independent sets. We say that two (three or four) vertices are \textit{consecutive} neighbours of $v$ if  they are its neighbours in their natural cyclic order around $v$ (see  Diestel \cite{flobar2}, Proposition $5.1.2$). Assume that vertices of  $V(J)$ satisfy  the following conditions:
\begin{enumerate}
\item[($a_{1}$)] every vertex of $J_{1} \cup J_{2}$ is of degree at most $4$,
\item[($a_{2}$)]  if $v \in J_1$ and $v_1$,  $v_2$,  $v_3$,  $v_4$ are its consecutive neighbours such that $v_1$, $v_2 \in J_3$ and $v_3$, $v_4 \in J_2$, then $v_1$ or $v_2$ is of degree at most $4$.
\end{enumerate}

In this Chapter we will define a \textit{preliminary} $3$-colouring of $V(J)$ (non proper) arisen from ${\cal J}$ (see Definition \ref{definition4.1}). Next, we will define a $3$-colouring $t\colon V(J) \rightarrow \{\alpha, \beta,\gamma\}$ \textit{associated} with ${\cal J}$  (see Definition~\ref{definition4.3}) and a \textit{good} $3$-colouring of $V(J)$  (see Definition~\ref{definition4.5}) such that the set of all vertices coloured by $\alpha$ induces a tree in $J$ and  each vertex of degree $4$ is coloured by $\gamma$ if and only if it has exactly two non-consecutive neighbours belonging to  $J_3$. Moreover, the set of all vertices of $J_{1} \cup J_{2}$ coloured by $\alpha$ is independent in $J$. 

We generate  a sequence of good $3$-colourings of $V(J)$ (beginning from the preliminary colouring) which successively have smaller numbers of induced cycles coloured by $\beta$ in $J$  (see Lemma \ref{lemma4.3} and Lemma \ref{lemma4.6}). At every step we change colour $\beta$ and $\alpha$ of at  most three vertices of $V(J)$. Finally, we obtain a good $3$-colouring such that the set of all vertices coloured by $\beta$  induces a tree in $J$. It plays an essential role  in the proof of Theorem \ref{theorem1.2}. 

\begin{definition}\label{definition4.1}
Let $X_{1}$,  $X_{2}$, $X_{3}$ be subsets of $J_{1} \cup J_{2}$ defined successively in the following way:
\begin{enumerate}
\item[$ $] $v \in X_{1} \Leftrightarrow v \in J_{1}$ and it has at most one neighbour belonging to $J_3$,
\item[$ $] $v \in X_{2} \Leftrightarrow v\in J_{2}$,  it  has at most one neighbour belonging to $J_3$ and its any neighbour does not belong to $X_1$,
 \item[$ $] $v \in X_{3} \Leftrightarrow v$ is a vertex of degree $4$ having exactly two non-consecutive neighbours belonging to~$J_3$.
\end{enumerate}
We define the following colouring $p\colon V(J) \rightarrow \{\alpha, \beta, \gamma\}$ $($non proper$)$ called a \textit{preliminary} colouring of~ $V(J)$. We put
\begin{enumerate}
\item[$ $]
$p(v) = \alpha$ for $ v \in J_{3}\cup X_{1} \cup X_{2}$,
\item[$ $]
$p(v) = \beta$ for $v \in (J_{1}\cup J_{2}) \setminus  (X_{1} \cup X_{2}\cup X_{3})$,
 \item[$ $]
$p(v) = \gamma$ for $v \in  X_{3}$.
\end{enumerate}
\end{definition}

\begin{definition}\label{definition4.2}
Let $t\colon V(J) \rightarrow \{\alpha, \beta, \gamma\}$ be a colouring of $V(J)$. A subgraph of $J$ is called \textit{${\alpha}_t$-subgraph} (\textit{${\beta}_t$-subgraph}, \textit{${\gamma}_t$-subgraph}) if $t(v)=\alpha$   ($t(v) =\beta$, ${t(v) =\gamma}$, respectively) for every vertex $v$ of this subgraph. 

We say that a cycle is \textit{big} if it has at least six vertices. A $\beta_t$-cycle is called  \textit{independent} if it is disjoint with any other $\beta_t$-cycle. Let $B_t$ denote the set of all $\beta_t$-cycles contained in $J$.
\end{definition}

\begin{definition}\label{definition4.3}
We say that a colouring  $t\colon V(J)\rightarrow \{\alpha, \beta,\gamma\}$ (non proper) is \textit{associated} with ${\cal J} =  \{J_{1}, J_{2}, J_{3}\}$ if the following conditions are satisfied:
\begin{enumerate}
\item[($1$)] every $\beta_t$-cycle is a $\beta_p$-cycle,
\item[($2$)]  every $\alpha_t$-vertex of $J_1$ of degree $3$ (or $4$) has at most one neighbour (two neighbours, respectively) in $J_3$,   
\item[($3$)] if an $\alpha_t$-vertex of $J_1$ has two neighbours in $J_3$, then each of its $\beta_t$-neighbours does not belong to any $\beta_t$-cycle, 
\item[($4$)] every $\alpha_t$-vertex of $J_2$ has at most one neighbour in $J_3$,
\item[($5$)] every vertex of $J_3$
is an $\alpha_t$-vertex,
\item[($6$)] a vertex of $V(J)$ is a $\gamma_t$-vertex if and only if it is a   vertex of degree $4$ having exactly two non-consecutive neighbours belonging to $J_3$, 
\item[($7$)] the set of all $\alpha_t$-vertices of $J_{1} \cup J_{2}$  
is independent in $J$,
\item[($8$)] the set of all $\alpha_t$-vertices of $V(J)$ induces a forest.
\end {enumerate}
\end{definition}

Remark that the preliminary colouring of $V(J)$ is associated with ${\cal J}$  because the set of all $\alpha_p$-vertices induces a forest in $J$ which has no path of length $3$ (i.e., it is a ``star forest''). 

Notice that if  $t$ is associated with ${\cal J}$ then, by condition $(1)$  of Definition \ref{definition4.3}, every $\beta_t$-cycle is a $\beta_p$-cycle. Hence, each vertex of any $\beta_t$-cycle does not belong to  $X_{1} \cup X_{2}\cup X_3$.

\begin{definition}\label{definition4.4}
Let $C$ be an induced cycle in $J$. If $v\in V(C)$, then $d_{ext(C)}(v)$ ($d_{int(C)}(v)$ ) is the number of edges from $v$ to a vertex contained in $ext(C)$ ($int(C)$, respectively).
\end{definition}

\begin{remark}\label{remark4.1}
Whitney proved that any two planar embeddings of a $3$-connected planar graph are equivalent (Diestel \cite{flobar2}, theorem~4.3.2). 
Hence, if $C$ is an induced cycle in $J$ and b is any vertex not belonging to $V(C)$, then, we can assume that $b \in ext(C)$.
\end{remark}

\begin{remark}\label{remark4.2} Let $C = abcda$ be an induced cycle in $J$. Since $J$ is $3$-connected the following implications hold:
\begin{enumerate}
\item[($a$)] if $d_{int(C)}(a) = 0$, then $d_{ext(C)}(c)  \neq 0$,
\item[($b$)] if $d_{int(C)}(a) = d_{int(C)}(c) = 0$, then $C$ is a facial cycle.
\end{enumerate}
\end{remark}
\begin{lemma}\label{lemma4.1}  Let $t$ be a colouring of $V(J)$ associated with  ${\cal J} =  \{J_{1}, J_{2}, J_{3}\}$.
\begin{enumerate}
\item[($a$)] If $v \in J_1$ is a vertex of a $\beta_t$-cycle, then it  has two consecutive neighbours in $J_3$.
\item[($b$)] If $v \in J_2$ is a vertex of a $\beta_t$-cycle, then it  has a neighbour belonging to $X_1$ or it has two consecutive neighbours in $J_3$.
\item[($c$)] Every $\beta_t$-cycle of length $4$ is a facial cycle.
\item[($d$)] Every two different $\beta_t$-cycles  of length $4$ are disjoint.
\end{enumerate}
\end{lemma}

\begin{proof}
Assume that $v \in J_1$ is a vertex of a $\beta_t$-cycle. Since  $v \notin X_{1}$, it  has two neighbours belonging to~$J_3$ (say $v_1$ and $v_2$). Hence, by condition $(a_1)$, $v$ is of degree $4$. Because $v \notin X_3$, $v_1$ and $v_2$ are consecutive neighbours of $v$. Hence, condition  ($a$) holds.

Suppose that $v \in J_2$ is a vertex of a $\beta_t$-cycle. Since $v \notin X_{2}$, it  has two neighbours belonging to $J_3$ or it has a neighbour belonging to $X_1$. Assume that $v$  has two neighbours belonging to $J_3$ (say $w_1$ and $w_2$). Hence, by condition $(a_1)$, $v$ is of degree $4$. Since $v \notin X_3$,  $w_1$ and $w_2$ are consecutive neighbours of $v$. Hence, condition  ($b$) holds.
 
Let $C = abcda$ be a $\beta_t$-cycle and suppose that $a$, $c \in J_1$. Since $J_2$ is an independent set in $J$, by condition $(a)$, $C$ is an induced cycle. In view of Remark \ref{remark4.1} and  condition $(a)$ we can assume that $d_{int(C)}(a)= 0$. Hence, by Remark \ref{remark4.2}$(a)$, $d_{int(C)}(a)  = d_{int(C)}(c) = 0$. Thus, by Remark \ref{remark4.2}$(b)$, $C$ is a facial cycle and condition ($c$) holds. 

Assume that there exist two different but not disjoint $\beta _t$-cycles of length $4$. By condition ($c$) they are facial cycles. If  they have at most two common vertices, then, by  conditions $(a)$ and $(b)$, we obtain a~contradiction. If they have three common vertices, then one of them has a vertex of degree $2$ which is a contradiction. Thus, condition ($d$) holds.
\end{proof}

\begin{lemma}\label{lemma4.2}
 Let $t$ be a colouring of $V(J)$ associated with ${\cal J}$ and suppose that $C$ is a big $\beta_t$-cycle. If $cde$ is a $\beta_t$-path, $c \in V(C)$ and $d \notin V(C)$, then  $c \in J_2$, $e \in V(C)$ and $cde$ is a path of a facial $\beta_t$-cycle.
 \end{lemma}
 
 \begin{proof}
Let $C$ be a big $\beta_t$-cycle and suppose that $cde$ is a $\beta_t$-path, $c \in V(C)$ and $d \notin V(C)$. By Remark~\ref{remark4.1}, we can assume that $d \in ext(C)$. Since $d \notin V(C) \cup J_3$, by Lemma \ref{lemma4.1}$(a)$, $c \in J_2$. Hence, $d \in J_{1}$ and $e \in J_2$. Let $xcy$ be the path contained in $C$. Then, $x, y \in J_1$.  Notice that, by condition $(a_1)$, $d_{ext(C)}(c) \leqslant  2$. Assume first that $d_{ext(C)}(c) = 1$. Since vertices $d$,  $x \in J_1$ ($d$, $y \in J_{1}$) are consecutive neighbours of $c$, there exists a facial $4$-cycle $C_{1} = cdx_{1}xc$ ($C_{2} = cdy_{1}yc$, respectively). Notice that, $x_{1} \notin J_3$ or $y_{1} \notin J_3$. Otherwise,  $d$ has two non-consecutive neighbours in $J_3$ and, by condition $(6)$ of Definition \ref{definition4.3}, $t(d) = \gamma$, contrary to $t(d) = \beta$.  If  $x_{1} \notin J_3$, then $x_{1}\in V(C)$ because, by  Lemma \ref{lemma4.1}$(a)$, $x$ has two neighbours of~$J_3$ and two neighbours of $V(C)$. Hence, $C_{1}$ is a $\beta_t$-facial cycle. Thus $e = x_{1}$ because, by  Lemma \ref{lemma4.1}$(a)$, $d$ has two neighbours of $J_3$ and two neighbours of $V(C_1)$. Similarly, if $y_{1} \notin J_3$, then $y_{1} \in V(C)$. Hence, $C_2$ is a $\beta_t$-facial cycle and $e = y_{1}$.

Suppose now that $d_{ext(C)}(c) = 2$. We can assume that vertices $d$, $x \in J_1$ are consecutive neighbours of $c$. Then, there exists a facial $4$-cycle $C_{1} = cdx_{1}xc$. Notice that if $x_{1} \notin V(C)$, then $x_{1} \in ext(C)$. Hence, $d_{ext(C)}(x) = 2$ because $x \in J_1$. Since $d_{int(C)}(c) = d_{int(C)}(x) = 0$, the cycle $C$ has a chord which is a~contradiction because $C$  is an induced cycle in $J$. Thus,  $x_1 \in V(C)$ and $C_{1}$ is a  $\beta_t$-facial cycle. Hence, $e = x_{1}$ because, by  Lemma \ref{lemma4.1}$(a)$, $d$ has two neighbours of $J_3$ and two neighbours of $V(C_1)$. 
\end{proof}

 \begin{corollary}\label{corollary4.1} If $C$ is a big $\beta_t$-cycle and $P$ is a $\beta_t$-path which has a common vertex with $C$, then each vertex of $ P$ belonging to $J_2$ is a vertex of $C$.
  \end{corollary}
  
\begin{definition}\label{definition4.5} Let $t$ be a colouring of $V(J)$ associated with ${\cal J} = \{J_{1}, J_{2}, J_{3}\}$.  An  $\alpha_t$-path with end vertices belonging to $J_3$ is called \textit{$t$-bad} in $J$ if each of its end-vertices is adjacent to a vertex of $J_1$ belonging to a~$\beta_t$-cycle. Every vertex of a bad path belonging to~$J_1$ is called \textit{$t$-bad}. The colouring  $t$ is called a \textit{good colouring} if $J_1$ contains at most one $t$-bad vertex.  
\end{definition}

\begin{remark} \label{remark4.3} From conditions $(4)$, $(5)$ and $(7)$ of Definition \ref{definition4.3}, we see that vertices of a bad path belong to $J_1 \cup J_3$. Moreover, by conditions $(2)$ and $(6)$ of Definition \ref{definition4.3},  each bad vertex has two consecutive neighbours of $J_3$ and two neighbours of $J_2$. 
\end{remark}

\begin{lemma}\label{lemma4.3}  Let $t$ be a colouring of $V(J)$ associated with ${\cal J}$ such that $J_1$ has no $t$-bad vertex. If there exists a non-independent $\beta_t$-cycle, then there exists a colouring $m$ of $V(J)$ associated with ${\cal J}$ such that~$J_1$ has no $m$-bad vertex and $B_m \subset B_t$.
\end{lemma}

\begin{proof} 
Let $C$ be a non-independent $\beta_t$-cycle. Since,  by Lemma \ref{lemma4.1}$(d)$, every two $\beta_t$-cycles of order $4$ are disjoint, we may assume that $C$ is a big cycle. From Lemma \ref{lemma4.2}, it follows that there exists a facial $\beta_t$-cycle $D$ and a vertex $c \in J_2$ which is a common vertex of $V(C)$ and $V(D)$.  Since $c \in J_2$, there exists a path $abc$ contained in $C$ such that $a, b \notin V(D)$.  By Lemma~\ref{lemma4.1}$(a)$, $a$, $c\in J_2$ are consecutive neighbours of $b \in J_1$. Hence,  there exists a facial $4$-cycle $abcqa$. Moreover, by condition $(6)$ of Definition \ref{definition4.3}, by Lemma \ref{lemma4.1}$(d)$ and by Lemma \ref{lemma4.1}$(b)$, $q$  is an $\alpha_t$-vertex belonging to $X_{1}$. In view of Remark \ref{remark4.1} we can assume that $q \in int(C)$. Certainly we can assume that $D = cdefc$, where $b$, $d$, $f$, $q$ are consecutive neighbours of~$c$. Since $q$,  $f \in J_1$ are consecutive neighbours of~$c$, there exists a facial $4$-cycle $qcfgq$. Notice that, by Lemma \ref{lemma4.1}$(a)$, $g \in J_3$. 

If $q$ is a vertex of degree $4$, we can assume that $a$, $c$,  $g$,  $h$ are consecutive neighbours of $q$. Then $h \in J_2$, because $q \in X_1$ and $g \in J_3$. Moreover, by condition $(7)$ of Definition \ref{definition4.3},   $h$ is not an $\alpha_t$-vertex. Since $a$, $h \in J_2$ are consecutive neighbours of $q$, there exists a facial $4$-cycle $qhjaq$,  where $j$ is a $\beta_t$-vertex or it is an $\alpha_t$-vertex belonging to $J_1 \cup J_3$. Hence, it suffices to consider the following cases:
\begin{enumerate}
\item[($a$)] $q$ is a vertex of degree $3$ or $h$ is a $\gamma_t$-vertex ($q$ is a vertex of degree $4$),
\item[($b$)] $h$ and $j$ are $\beta_t$-vertices ($q$ is a vertex of degree $4$),
\item[($c$)] $h$ is a $\beta_t$-vertex and $j$ is an $\alpha_t$-vertex ($q$ is a vertex of degree $4$). 
\end{enumerate}
Case $(a)$. We put $m(c) = \alpha$, $m(q) = \beta$ and $m(v) = t(v)$ for $v \notin \{c, q\}$.  Notice that $q$ has only one $\beta_m$-neighbour (the vertex $a$). Thus, $B_m \subset B_t$.  Therefore $J_1$ has no $m$-bad vertex.  Since all neighbours of~$c$ are $\beta_m$-vertices, conditions $(7)$--$(8)$ of Definition \ref{definition4.3} (for $m$) are satisfied. Hence, $m$ is a colouring of $V(J)$ associated with ${\cal J}$.

Case $(b)$. Hence, $hjabc$ is a path contained in $C$ and, by Lemma \ref{lemma4.1}$(a)$,  both $j$ and $b$ have two neighbours belonging to $ext(C) \cap J_3$. Suppose that $a$ has a neighbour (say $p$) belonging to $ext(C)$. If $p$ is an $\alpha_t$-vertex, then, by conditions $(2)$ and $(6)$ of Definition \ref{definition4.3}, $p \in J_{3}$.  We put $m(a) =m(c) = \alpha$, $m(q) = \beta$ and $m(v) = t(v)$ for $v \notin \{a, c, q\}$.   Notice that $q$ has  only one $\beta_m$-neighbour (the vertex $h$). Thus, $B_m \subset B_t$. Therefore~$J_1$ has no $m$-bad vertex. Since $a$ has at most one $\alpha_m$-neighbour (the vertex $p \in J_3$),  conditions~$(4)$ and $(7)$--$(8)$ of Definition \ref{definition4.3} (for $m$) are satisfied. Hence, $m$ is a colouring of $V(J)$ associated with ${\cal J}$. 

Case $(c)$. We put $m(c) = \alpha$, $m(q) = \beta$ and $m(v) = t(v)$ for $v\notin \{c, q\}$. Notice that conditions $(2)$--$(8)$ of Definition \ref{definition4.3} (for~$m$) are satisfied. Assume that $q$ belongs to a $\beta_m$-cycle. Then, there exists a $\beta_t$-path from $h$ to a vertex of $V(C)$. Since $h \in J_2$, by Corollary~\ref{corollary4.1}, $h \in V(C)$.  Since $q, j \notin V(C)$, there exists a vertex $k \in V(C) \cap J_1$ such that  $hqgkh$ is a facial $4$-cycle. Notice that $d_{int(C)}(k) = 2$ because $k$ is adjacent to $g \in int(C)$ and, by Lemma \ref{lemma4.1}$(a)$, it has two consecutive neighbours of $J_3$. Further, $d_{int(C)}(h) = 2$ because $h$ is adjacent to vertices $j$, $q \in int(C)$.  Since $d_{ext(C)}(k) = d_{ext(C)}(h) = 0$, $C$ has a chord which is a contradiction because $C$ is an induced cycle in $J$. Hence, $q$ doesn't belong to any $\beta_m$-cycle. Therefore, $B_m \subset B_t$. Certainly, $J_1$ has no $m$-bad vertex.
\end{proof}

\begin{lemma}\label{lemma4.4} 
Let $t$ be a good colouring of $V(J)$ and suppose that $C$ is an independent $\beta_t$-cycle.  If  there exists a $t$-bad path (say $bad$) we choose a vertex $w \in V(C)\cap J_1$ such that $abwda$ is not a facial cycle.  If~$J_1$ has no $t$-bad vertex, then $w$ is any vertex of $V(C)\cap J_1$. We recolour the vertex $w$, swapping $\beta$ with~$\alpha$. Then, we obtain a new colouring $m$ of $V(J)$ satisfying the following conditions:
\begin{enumerate}
\item[($a$)] $m$ is a colouring associated with ${\cal J}$, 
\item[($b$)] if $a$ is a $t$-bad vertex, than any vertex different from $a$ and $w$ is not an  $m$-bad vertex,
\item[($c$)] if  $J_1$ has no $t$-bad vertex, then  any vertex different from $w$ is not an  $m$-bad vertex.
\end{enumerate}
\end{lemma}

\begin{proof} Assume that $bad$  is a $t$-bad path. Since, by Remark \ref{remark4.3}, $b$, $d$ are consecutive neighbours of $a$, there exists a facial cycle $abcda$. Let $C$ be an independent $\beta_t$-cycle and suppose that $w$ is a vertex of $V(C)\cap J_1$, $w \neq c$. 

We first prove $(a)$ (the same proof goes when $J_1$ has no $t$-bad vertices). Notice that each $\alpha_m$-vertex ($\beta_m$-vertex) different from $w$ is  an $\alpha_t$-vertex ($\beta_t$-vertex, respectively).  From this it follows that every $\beta_m$-cycle is a $\beta_t$-cycle and each $\beta_m$-neighbour of $w$ does not belong to any $\beta_m$-cycle because $C$ is an independent $\beta_t$-cycle.  It is easy to check that the colouring $m$ satisfies all conditions $(1)$--$(7)$ of Definition~\ref{definition4.3}. Further, every $\alpha_m$-cycle omitting the vertex $w$ is an $\alpha_t$-cycle. Therefore, it suffices to prove that $w$ does not belong to any $\alpha_m$-cycle.

Suppose, by contradiction, that there exists an $\alpha_m$-cycle (say $D$) containing the vertex $w$. Hence $D - w$ is an $\alpha_t$-path with end-vertices belonging to $J_3$. Then, $D - w$ is a $t$-bad path, because its end-vertices are adjacent to the vertex $w \in V(C)$.  Since $t$ is a good colouring, $D - w = bad$. Since $w \in J_1$ is a vertex of the cycle $C$ and $b, d \in J_3$ are neighbours of $w$, then by Lemma \ref{lemma4.1}$(a)$, $b, d$ are consecutive neighbours both of $w$, $a$ and $c$, which by Remark \ref{remark4.1}$(a)$, is a contradiction (consider cycles $abcda$ and $abwda$). 

We now prove $(b)$ and $(c)$. It suffices to prove that if a vertex different from $w$  is $m$-bad, then it is $t$-bad. Let $Q$ be an $m$-bad path.  If $w \notin V(Q)$, $Q$ is a $t$-bad path because  each $\alpha_m$-vertex different from $w$ is  an $\alpha_t$-vertex and every $\beta_m$-cycle is a $\beta_t$-cycle. If $w \in V(Q)$, then $Q - w$ is the union of two $\alpha_t$-paths (say $Q_1$ and $Q_2$).  Hence, $Q_1$ and $Q_2$ are $t$-bad paths because $w \in V(C)$ is an end-vertex both of $Q_1$ and $Q_2$.
 \end{proof}
 
 \begin{lemma}\label{lemma4.5} Let $t$ be a colouring of $V(J)$ associated with ${\cal J}$ and suppose that each $\beta_t$-cycle is independent. Let $a \in J_1$ be an $\alpha_t$-vertex which has two consecutive neighbours of $J_3$ (say $b$ and $d$) and two neighbours of $J_2$. If $b$ is of degree at most $4$, then there exists at most one $\beta_t$-cycle (say $C$) which contains a neighbour of $b$ belonging to $J_1$. Moreover, $C$ contains at most two neighbours of $b$ belonging to $J_1$. Further, if $C$ is a $4$-cycle, then it contains only one neighbour of $b$ belonging to $J_1$.
\end{lemma}

\begin{proof}
Let $a \in J_1$ be an $\alpha_t$-vertex which has two consecutive neighbours belonging to $J_3$ (say $b$ and~$d$) and two neighbours of $J_2$.  Then there exists a facial cycle $abcda$. Suppose that $a$, $c$, $c_{1}$, $c_{2}$ are consecutive neighbours of $b$ (the same proof goes when $b$ is of degree $3$).  

Assume first that $c_{2}$ is a vertex of $J_{1}$ belonging to a $\beta_t$-cycle (say $C_2$). Since $c_{2}$, $a \in J_{1}$ are consecutive neighbours of $b$, there exists a facial $4$-cycle $c_{2}baec_{2}$. If $e \notin J_3$ then  $e \in V(C_{2})$ because, by Lemma~\ref{lemma4.1}$(a)$,~$c_2$ has  two neighbours of $J_3$ and two neighbours of $V(C_2)$. Hence, $a$ has a neighbour  belonging to $V(C_2)$ which, by condition $(3)$ of Definition \ref{definition4.3}, is a contradiction. If $e \in J_3$, then $a$ has three neighbours belonging to $J_3$ which is a contradiction. It follows that $b$ has at most two neighbours in $J_1$  (different than $c_2$ and $a$) each of which belongs to a $\beta_t$-cycle.  

Let now $c$, $c_{1} \in J_{1}$ and each of them belongs to a $\beta_t$-cycle (say $C$ and $C_1$, respectively).  Since $c$ and~$c_1$ are consecutive neighbours of $b$ there exists a facial $4$-cycle $D = cfc_{1}bc$. Hence,   by Lemma~\ref{lemma4.1}$(a)$, $f \in V(C)$ because $b, d$ are  neighbours of $c$ belonging to $J_3$. Since $f \notin J_3$, $f \in V(C_1)$ because, by Lemma~\ref{lemma4.1}$(a)$, $c_1$ has  two neighbours of $J_3$ and two neighbours of $V(C_1)$. Therefore, $C = C_1$ because they are independent. Remark that $C$ is a big cycle. Otherwise, by Lemma~\ref{lemma4.1}$(c)$, $C$ is a facial cycle. Then, $C$ and $D$ are different facial cycles with a common path $cfc_1$ which is a contradiction.
\end{proof}

\begin{lemma}\label{lemma4.6}
 Let $t$ be a good colouring of $V(J)$ and suppose that each $\beta_t$-cycle in $J$ is independent.  If  $B_{t} \neq \emptyset$, then there exists a good colouring $m$ of $V(J)$ such that $B_{m} \subset B_{t}$.
\end{lemma}
\begin{proof} Assume first that there exists a $t$-bad vertex (say $a$). By Remark \ref{remark4.3}, $a$ has two consecutive neighbours belonging to $J_3$ (say $b$ and $d$) and two neighbours of $J_2$.  In view of condition $(a_2)$ we can assume that $b$ is of degree at most $4$. Since $t$ is a good colouring of $V(J)$, $bad $ is a $t$-bad path. Hence, $b$ is adjacent with a vertex of $J_1$ belonging to a $\beta_t$-cycle (say $C$).  If $C$ is a big cycle ($4$-cycle), then by Lemma \ref{lemma4.5}, there are at most two vertices (only one vertex, respectively) belonging to $V(C) \cap J_1$ which are (is, respectively) adjacent to $b$. Hence, $C$ contains a vertex $w \in J_1$ which is not adjacent to $b$. 

We put $m(w) = \alpha$ and $m(v) = t(v)$ for other vertices of $V(J)$. 
From Lemma \ref{lemma4.4}$(a)$, we see that $m$ is a colouring associated with $\cal J$. We will prove that $a$ is not an $m$-bad vertex. Suppose by contradiction, that $a$ belongs to an $m$-bad path (say $Q$). 
From Lemma~\ref{lemma4.4}$(b)$, it follows that any vertex of $J_1$ different from $a$ and $w$ does not belong to $Q$. Hence, $b$ is an end-vertex of $Q$ because it is adjacent to $a$ and it is not adjacent to $w$. Therefore, $b $ is adjacent to a vertex of $J_1$ belonging to a $\beta_m$-cycle (say $D$). Since~$D$ is a $\beta_t$-cycle, by Lemma  \ref{lemma4.5}, $C = D$. This is a contradiction because  $C$ is not a $\beta_m$-cycle. Therefore, $a$ is not an $m$-bad vertex. Hence,~$m$ is a good colouring of $V(J)$ such that $B_{m} \subset  B_{t}$. 

Assume now that $J_1$ has no $t$-bad vertices. Since $B_{t} \neq \emptyset$, there exists an independent $\beta_t$-cycle (say~$C_{0}$). We choose any vertex of $V(C_{0}) \cap J_1$ (say $w$).  We put $m(w) = \alpha$ and $m(v) = t(v)$ for other vertices of~$V(J)$. Hence, Lemma \ref{lemma4.6} follows from Lemma \ref{lemma4.4}$(a)$ and~\ref{lemma4.4}$(c)$. 
\end{proof}

  \section{Final results}

Let $J \in {\Theta}$ and suppose that  ${\cal J} = \{J_{1}, J_{2}, J_{3}\}$ is a partition of $V(J)$ into three independent sets.  Recall, that vertices belonging to~$J_{i}$ are called of type $i$. Now we are ready to prove Theorem \ref{theorem1.2}.

\begin{proof} [\bf Proof of Theorem \ref{theorem1.2}]

Notice that the preliminary colouring $p$ of $V(J)$ has no $p$-bad vertices.  In view of Lemma \ref{lemma4.3} and  Lemma \ref{lemma4.6}, there exists a good colouring $t$ of $V(J)$ such that $B_{t} = \emptyset$.

Suppose that $\{v_1, \ldots, v_n\} \subseteq J_{1} \cup J_{2}$ is a sequence of all $\gamma_t$-vertices of $V(J)$. We define by induction a sequence $\{t_{0}, t_{1}, \ldots, t_{n}\}$  of $3$-colourings of $V(J)$ satisfying the following conditions: $t_{0} = t$ and 
\begin{enumerate}
\item[($a$)] $t_{i}(v_{i}) = \alpha$ ($t_{i}(v_{i}) = \beta$) if there exists (does not exists, respectively) a $\beta_{t_{i-1}}$-path connecting two neighbours of $v_i$, for every $1 \leqslant  i \leqslant n$,
\item[($b$)] if $v \neq \ v_{i}$, then $t_{i}(v) = t_{i-1}(v)$, for every $1 \leqslant  i \leqslant n$.
\end{enumerate}
We first prove that
\begin{enumerate}
\item[($c$)] the set of all $\alpha_{t_i}$-vertices of $J_{1} \cup J_{2}$ is independent in $J$, for  $0 \leqslant  i \leqslant n$.
 \end{enumerate}
We apply induction on $i$. From condition $(7)$ of Definition \ref{definition4.3}, we see that $(c)$ holds for $i = 0$. Notice that~$v_{i+1}$ is a vertex of degree~$4$ and it has two neighbours in $J_3$. Hence, if $t_{i+1}(v_{i+1}) = \alpha$, then, by $(a)$, it is not adjacent to any $\alpha_{t_{i}}$-vertex of $J_{1} \cup J_{2}$. Moreover, from $(b)$, we conclude that a vertex $v \neq v_{i+1}$ is an $\alpha_{t_{i+1}}$-vertex  of $J_{1} \cup J_{2}$ if and only if its an $\alpha_{t_{i}}$-vertex of $J_{1} \cup J_{2}$. Thus, if $(c)$ holds for $i < n$, then it holds for $i+1$.

We now prove that
\begin{enumerate}
\item[($d$)] the set of all $\alpha_{t_i}$-vertices ($\beta_{t_i}$-vertices) induces a forest in $J$, for  $0 \leqslant  i \leqslant n$.
 \end{enumerate}
We apply induction on $i$. Since $B_{t} = \emptyset$, condition $(8)$ of Definition \ref{definition4.3} shows that $(d)$ holds for $i = 0$.  From~$(b)$, we conclude that a vertex $v \neq v_{i+1}$ is an $\alpha_{t_{i+1}}$-vertex ($\beta_{t_{i+1}}$-vertex) if and only if it is an $\alpha_{t_{i}}$-vertex ($\beta_{t_i}$-vertex, respectively). Notice that, by condition $(a)$ and by the Jordan curve theorem, $v_{i+1}$ belong neither to any $\alpha_{t_{i+1}}$-cycle nor to  any $\beta_{t_{i+1}}$-cycle.  Thus, if~$(d)$ holds for $i < n$, then it holds for $i+1$.

From $(c)$ it follows that the set $I$ of all $\alpha_{t_n}$-vertices of $J_{1} \cup J_{2}$ is independent in $J$. Notice that $J_{3} \cup I$ ($(J_{1} \cup J_{2})\setminus  I$) is the set of all $\alpha_{t_n}$-vertices ($\beta_{t_n}$-vertices, respectively) of  $V(J)$. Hence, by $(d)$, the theorem holds.
\end{proof}

\begin{proof} [\bf Proof of Theorem \ref{theorem1.3}]

We apply induction on the order of graphs belonging to $\Theta$.  If $J$ is  $3$-connected,  then,  by Theorem \ref{theorem1.2} and  by Lemma \ref{lemma1.1}, the theorem holds. Similarly, if $J$ is a $3$-cycle or a $4$-cycle, then the theorem holds. 

Suppose that $J$ is not $3$-connected and $J$ is neither $3$-cycle nor $4$-cycle. Then $J$ is a union of two subgraphs (say $G$ and $G_0$)  having only two common vertices (say $b, d$), and each of these subgraphs is different from the edge $bd$. 

Assume first that $J$ has no vertex of degree $2$.  It is easy to check that  both $G$ and $G_0$ are not bounded by a $3$-cycle. Thus, both~$G$ and $G_0$ are bounded by a $4$-cycle (say $C = abcda$ and $C_{0} = a_{0}bc_{0}da_{0}$, respectively) and vertices $b, d$ are of degree~$4$. Hence, at least one of the following conditions is satisfied:
\begin{enumerate}
\item[($a$)] $ V(G) \setminus  V(C)$ is contained in the interior of $C$, 
\item[($b$)] $V(G_0) \setminus  V(C_0)$ is contained in the interior of $C_0$.
\end{enumerate}
Assume that condition $(a)$ occurs. Hence, 
\[
d_{int(C)}(b) = d_{int(C)}(d) = d_{ext(C)}(a) = d_{ext(C)}(c) = 0.
\]  
Certainly, there exists a  cycle (say $C_{1} =  a_{1}b_{1}c_{1}d_{1}a_{1})$ contained in $G$ with the analogical property 
\[
d_{int(C_{1})}(b_{1}) = d_{int(C_{1})}(d_{1}) = d_{ext(C_{1})}(a_{1}) = d_{ext(C_{1})}(c_{1}) = 0
\]
 such that  $int(C_1)$ is \textit{minimal}. It means that not exist any $4$-cycle $E$ contained in $G$ with the same property such that $int(E) \subset int (C_{1})$. Since $int(C_1)$ is minimal, it is easy to check that vertices $a_{1}$, $c_{1}$ are of degree~$3$ in $J$ and they are adjacent. Certainly, $J -\{a_{1}, c_{1}\} \in {\Theta}$ and ${\Delta}(J -\{a_{1}, c_{1}\}) \leqslant 4$. By the induction hypothesis, there exists a partition ${\cal F}$ of  $V(J) \setminus  \{a_{1}, c_{1}\}$ into two subsets which is compatible with the partition $\{J_{1} \setminus  \{a_{1}, c_{1}\}, J_{2} \setminus  \{a_{1}, c_{1}\}, J_{3} \setminus  \{a_{1}, c_{1}\}\}$ and each part of ${\cal F}$ induces a forest in $J -\{a_{1}, c_{1}\}$. Then, by Lemma \ref{lemma3.1},  there exists a partition~$\cal K$ of $V(J)$  into two subsets which is compatible with~$\cal J$ and each part of $\cal K$  induces a forest in $J$. 
 
Assume now that $J$ has a vertex of degree $2$. Since $J$ is neither a $3$-cycle nor a $4$-cycle the following condition is satisfied
\begin{enumerate}
\item[($c$)] there exists a $k$-cycle $D$, for some $k = 3$,  $4$, and there is a vertex $a$ of degree $2$ which is the only vertex belonging to $int(D)$ or $ext(D)$. 
\end{enumerate}
 Then, $J - a \in {\Theta}$ and ${\Delta}(J -a) \leqslant 4$. By the induction hypothesis, there exists a partition ${\cal F}$ of  $V(J) \setminus  \{a\}$ into two subsets which is compatible with $\{J_{1}  \setminus  \{a\}, J_{2}  \setminus  \{a\}, J_{3}  \setminus  \{a\}\}$ and each part of ${\cal F}$ induces a forest in $J - a$. Then, by Lemma \ref{lemma3.2},  there exists a partition~$\cal K$ of $V(J)$  into two subsets which is compatible with $\cal J$ and each part of $\cal K$  induces a forest in $J$.
\end{proof}

\end{document}